\documentclass{article}
\usepackage{ctex}
\usepackage{indentfirst}
\usepackage{geometry}
\usepackage{chngpage}

\usepackage{ulem}    
\usepackage{mathrsfs} 
\usepackage{amssymb,amsfonts,amsmath,amsthm,mathtools}
\usepackage{xcolor}
\usepackage{mdframed}
\mdfdefinestyle{shadeStyle}{backgroundcolor=white,linewidth=0pt,innerleftmargin=1ex,innertopmargin=-0.5ex}
\theoremstyle{plain}
\newmdtheoremenv[style=shadeStyle]{theorem}{Theorem}[section]
\newmdtheoremenv[style=shadeStyle]{lemma}[theorem]{Lemma}
\newmdtheoremenv[style=shadeStyle]{proposition}[theorem]{Proposition}
\newmdtheoremenv[style=shadeStyle]{corollary}[theorem]{Corollary}
\newmdtheoremenv[style=shadeStyle]{remark}[theorem]{Remark}

\renewcommand{\abstract}{Abstract.}
\theoremstyle{definition}
\newcommand{\rd}{\mathrm{d}}
\usepackage{amsmath}
\numberwithin{equation}{section}

\usepackage[T1]{fontenc}
\usepackage{authblk}
\ctexset{
	bibname = References
}
\begin{document}

	\title{\bf Suppression of  Blowup  by Slightly Superlinear  Degradation in a Parabolic-Elliptic Keller--Segel System with Signal-dependent Motility}
	\author[1,2]{Aijing Lu}
	\author[1,2]{Jie Jiang \thanks{Corresponding author: jiang@apm.ac.cn}}
	\affil[1]{		Innovation Academy for Precision Measurement Science and Technology, Chinese Academy of Sciences, Wuhan 430071, China}
	\affil[2]{University of Chinese Academy of Sciences, Beijing 100049, China}
	\date{ }
	\maketitle
	
	\begin{abstract}
		In this paper, we consider an initial-Neumann boundary value problem for a parabolic-elliptic Keller-Segel system with signal-dependent motility  and a source term. Previous research has rigorously shown that the source-free version of this system exhibits an infinite-time blowup phenomenon when dimension $N \geq 2$. In the current work, when $N \leq 3$, we establish uniform boundedness of global classical solutions  with an additional source term that involves  slightly super-linear degradation effect on the density, of a maximum  growth order $s\log s$, unveiling a sufficient blowup suppression mechanism. The motility function considered in our work takes a rather general form compared with  recent works \cite{FuJi2020, LyWa2023}  which were restricted to the monotone non-increasing case.  The cornerstone of our proof lies in deriving an upper bound for the second component of the system and an entropy-like estimate, which are achieved  through tricky comparison skills and energy methods, respectively.                                                                                                                                                                                                                                                                                                                                                            
	\end{abstract}                                                                                                                            	
	\section{Introduction}

	Chemotaxis is the phenomenon that cells or organisms move in a directed manner in response to chemical signals. The study of chemotaxis can be traced back to Patlak's work in 1950 \cite{Patlak1953}.  In 1970, Keller and Segel  proposed a classic chemotaxis model in   \cite{ Keller1970,1971KS}, which is described by the following equations, 
	\begin{equation}\label{0.2}
		\left\{
		\begin{aligned}
			& u_t = \nabla \cdot \left( \gamma(v)\nabla u-\chi (v)u\nabla v\right), \\
			&  \tau v_t=\Delta v-v+u.
		\end{aligned}
		\right.
	\end{equation}
	Here, $\tau \geq0$ is a given constant. $u$ and $v$ represent the density of cells, and the concentration of chemical signals, respectively. $\gamma(\cdot)$ is the motility coefficient of cells, while $\chi(\cdot)$ represents the sensitivity of cells to chemical signals. The relationship between $\gamma(\cdot)$ and $\chi(\cdot)$ is given by
	\begin{equation*}
		\chi(\cdot)=(\alpha-1)\gamma'(\cdot),
	\end{equation*}
	where $\alpha$ is a given non-negative constant representing a rescaled distance between signal receptors in cells.  When 
	$\alpha > 0$,  cells can perceive the concentration of chemical signals at different positions, allowing them to determine the   direction of cell movement. This mechanism is commonly referred to as gradient sensing. When $\alpha=0$, cells only possess one receptor,  limiting their ability to detect the chemical signal to a single point, which  is called as the local sensing mechanism \cite{BLT2020,1971KS}. In the latter case,  $\chi=-\gamma'$ and hence system \eqref{0.2} can be re-written as 
	\begin{equation}\label{0.3}
		\left\{
		\begin{aligned}
			& u_t = \Delta(\gamma(v)u),\\
			&  \tau v_t=\Delta v-v+u.
		\end{aligned}
		\right.
	\end{equation}
	In the above model, $\gamma'(v)<0$ corresponds to chemotactic attraction, where bacteria are attracted to regions of higher chemical signal concentration. Conversely, $\gamma'(v)>0$ corresponds to chemotactic repulsion, where bacteria move away from regions of higher chemical signal concentration.
	
	Recently, numerous research results have been obtained for the initial-Neumann boundary value problem of system \eqref{0.3}. Existence of globally bounded classical solutions was established in any dimension  $N\geq 1$, when either $\gamma$ has strictly positive upper and lower  bounds \cite{TW2017, XiaoJiang2022}, or $\tau=0$ and $\gamma$ is non-decreasing or unbounded \cite{JiLa2023}. If $\gamma$ is non-increasing and tends to zero at infinity, problem \eqref{0.3} may become degenerate as $v \rightarrow \infty$, and it was shown under such a  case  that classical solution exists globally  for $N \geq 1$ \cite{AhnYoon2019,FuJi2020,FuJi2021b,FuSe2022a,JiLa2021,JLZ2022,XiaoJiang2022,YoonKim17}. Moreover, the dynamic behavior of the  solutions was shown to be  closely related to the decay rate of  $\gamma$ at infinity. In particular, when $\gamma(v)= e^{-v}$,  an infinite-time blowup was rigorously proved  which is  distinct from the well-known finite-time blowup phenomenon for the classical Keller-Segel system \cite{FuJi2020, FuSe2022a}. For studies on  weak solutions of \eqref{0.3},  we refer the readers to \cite{BLT2020,Desvillettes2019,LiJi2021}.

	More recently, 	system \eqref{0.3} with an additional standard logistic source term as follows was proposed in \cite{FuLiu2012stripe}  to study the formation principles for various patterns in ecological systems:
	\begin{equation}\label{0.4}
		\left\{
		\begin{aligned}
			& u_t = \Delta(\gamma(v)u)+\mu (u-u^2), \\
			&  \tau v_t-\Delta v+v=u.
		\end{aligned}
		\right.
	\end{equation}
	Here, $\mu> 0$. For problem \eqref{0.4}, when $N=2$ and $\tau=0$, under a condition $\sup\limits_{0\leq s <\infty} \frac{|\gamma'(s)|^2}{\gamma(s)}<\infty$,  existence of globally bounded classical solution was established in \cite{Jin2020keller}, and the additional condition on $\frac{|\gamma'|^2}{\gamma}$ was later removed in \cite{FuJi2020}. When $N=2$ and $\tau >0$,  supposing that $\frac{|\gamma'(s)|}{\gamma(s)}<\infty$, existence of globally bounded classical solutions  was proved in \cite{JKW2018}.  We refer the readers to \cite{LiuXujiao2019,Tello2022,WangWang2018JMP} for studies in  higher-dimensional case $N \geq 3$, where sufficient largeness of $\mu$ was necessarily needed. On the other hand, generalized logistic-type sources as $\mu(u-u^\sigma)$ with some $\sigma>1$ were also considered in the literature. When $\tau=0$, uniform boundedness of globally classical solution was verified in \cite{LyWa2023},  when any of following holds: (i) $N\leq 2$, $\sigma>1$, (ii) $N\geq 3$, $\sigma>2$, or (iii) $\sigma =2$ and $\mu$  is sufficiently large. Existence of weak solutions for problem \eqref{0.4}  with a  more general source $uh(u)$ satisfying	$ (h(s)\log s) /s \rightarrow -\infty$ as $ s\rightarrow \infty$,  was shown  in \cite{Desvillettes2023weak} provided that $\sup_{s>s_0}(\gamma(s)+s|\gamma'(s)|^2/\gamma(s))<\infty$ for all $s_0> 0$.

	The motivation of this paper comes from an observation that when $\mu =0$ and $\gamma=e^{-v}$, the solution of problem \eqref{0.3} exhibits a blow-up phenomenon as $t\rightarrow \infty $ \cite{FuJi2020,FuJi2021a,FuSe2022a} for $N\geq2$. It is expected that  an external source term with a degradation effect on the density would lead to globally bounded solutions. Previous studies have demonstrated that the logistic source term $\mu (u-u^2)$ or its generalized form $\mu (u-u^\sigma)$ can prevent blow-up \cite{ FuJi2020,JKW2018,Jin2020keller,LyWa2023}, which, however, requires a strong dampening when $N\geq3$, by assuming that $\mu$ is sufficiently large in the standard logistic source, or a super-quadratic growth order $\sigma>2$ in the latter.
	The aim of this study is to determine the critical growth condition for an external source that can guarantee the boundedness of the classical solution, unveiling an adequate suppression mechanism on the blowup of density detected in the source-free version of \eqref{0.4}.
	
	More precisely, we study  the  following initial-Neumann boundary value problem in a smooth bounded domain $\Omega \subset \mathbb{R}^N$ with $N\geq1$.
	\begin{subequations}\label{0.1}
		\begin{align}
			& u_t = \Delta\big( u \gamma(v)\big)- u f(u), \qquad &(t,x) &\in  (0,\infty) \times\Omega, \label{35}\\
			&  v-\Delta v = u ,  \qquad &(t,x) &\in (0,\infty) \times\Omega,\label{36}\\
			& \nabla\big( u\gamma(v)\big)\cdot \mathbf{n} = \nabla v\cdot \mathbf{n} = 0, \qquad &(t,x) &\in (0,\infty) \times\partial\Omega, \label{37}\\
			& u(0) = u^{in}, \qquad &x &\in\Omega. \label{38}
		\end{align}
	\end{subequations}
	We assume that $u^{in}$ satisfies the following conditions
	\begin{equation}\label{uin}
		u^{in}\in W^{1,\infty}( \Omega ),~~~~u^{in}\geq0~,~~~~u^{in} \not \equiv 0\;\;\text{in}~\Omega,
	\end{equation}
	and for $\gamma(\cdot)$, we require that
	\begin{equation}\label{gamma}
		\gamma(\cdot)\in C^3([0,\infty)),~~~~\gamma(\cdot)>0,~~~~  \text{on}~[0,\infty).
	\end{equation}
	Besides, we assume that
	\begin{equation}\label{f1}
		f(\cdot)\in C^{1}([0,\infty)),~~~~	\lim_{s\rightarrow \infty}f(s)=\infty,
	\end{equation}
	and
	\begin{equation}\label{f2}
		\limsup\limits_{s\rightarrow \infty}\frac{f(s)}{\log s}< \infty.
	\end{equation}
	A typical example fulfilling assumptions \eqref{f1}-\eqref{f2} is that $f(s)=\lambda \log^\alpha(1+s)-\mu$, where $\lambda>0$, $\mu \in \mathbb{R}$ and $0<\alpha\leq 1$.
	\\
	\par The main result of the present work on  global boundedness of solutions is stated as follows.
	\begin{theorem}\label{thm1}
		Assume $\Omega \subset \mathbb{R}^N$ with $N\leq3$. Suppose that $f(\cdot)$ satisfies the conditions \eqref{f1}-\eqref{f2}, and that $ \gamma(\cdot)$  satisfies   \eqref{gamma}. For any given initial datum $u^{in}$ satisfying \eqref{uin}, problem \eqref{0.1} has a unique global non-negative classical solution  $(u,v)\in \left( C^0([0,\infty)\times \bar{\Omega})\cap C^{1,2}( (0,\infty) \times \bar{\Omega}) \right)^2$, which is uniformly-in-time bounded. More precisely, there is a  constant $C>0$ depending only on $u^{in}$, $\gamma$, $f$, $N$, and $\Omega$ such that 
		\begin{equation*}
			\|u(t,\cdot)\|_{L^\infty(\Omega)}+\|v(t,\cdot)\|_{L^\infty(\Omega)}\leq C,~~~~~~\text{for all}\;\;t>0.
		\end{equation*}
	\end{theorem}
	\par Complementing $\gamma$ with additional properties, more precisely, if $\gamma$ is non-decreasing and concave, we have  explicit upper bound estimates for $u$ in any dimensional spaces $N \geq 1$. Note we only need assumption \eqref{f1} on $f$ now.
	\begin{theorem}\label{thm2}
		Let $N \geq 1$.	Assume that $f$ satisfies \eqref{f1}, and $u^{in}$ satisfies \eqref{uin}.	If  $\gamma$ satisfies \eqref{gamma},  and moreover
		\begin{equation}\label{gamma2}
			\gamma' \geq 0,~~\gamma''\leq 0~~~~\text{on~~~~}[0, \infty),
		\end{equation}
		then, for any given initial datum $u^{in}$ satisfying \eqref{uin}, problem \eqref{0.1} has a unique global non-negative classical solution  $(u,v)\in \left( C^0([0,\infty)\times \bar{\Omega})\cap C^{1,2}( (0,\infty) \times \bar{\Omega}) \right)^2$, which is uniformly-in-time bounded. Moreover,
		\begin{equation*}
			\|u(t,\cdot )\|_{L^\infty(\Omega)}\leq \max \{ \|u^{in}\|_{L^\infty(\Omega)},\,\beta_1\}, ~~~~~ \text{for any~} t>0,
		\end{equation*}
		where   $\beta_1$ is a  non-negative constant depending only on $f$ given in Remark \ref{beta} below.
	\end{theorem}

	One key of our proof  lies in deriving an upper bound for the second component $v$ under a rather general assumption \eqref{gamma} on $\gamma$. According to  the different asymptotic behavior of $\gamma$ near infinity,  two distinct comparison approaches proposed recently in \cite{FuJi2020} and \cite{JiLa2023}, respectively, are carefully modified to tackle the nonlinear difficulty brought by the new super-linear source term. Note that uniform upper boundedness of $v$ is proved with any super-linear degradation in any dimension $N\geq1$ in the current contribution. With the uniform boundedness of $v$ at hand, we employ standard  energy method to show the uniform boundedness of $u$ when $N=2$. Moreover, we derive an entropy-like estimate involving $\int_\Omega u\log u\;\rd x$ in any dimensional spaces $N\geq 1$, which also plays a crucial rule in the proof, since it provides us an $L^1$-boundedness of  $uf(u)$ due to our assumption \eqref{f2}. Such an estimate then helps us in the proof of  H\"older continuity of $v$, and further application of semigroup theory to establish the uniform boundedness of $u$ when $N=3$.

	It is worth mentioning that, compared with previous research \cite{FuJi2020,JKW2018,LyWa2023}, we significantly relax the conditions on $\gamma$ by eliminating the requirement of non-increasing  monotonicity. Indeed, we only need $\gamma$ to satisfy \eqref{gamma}. In addition,  the external source $uf(u)$ is permitted to be of any slightly super-linear growth order not faster than $s\log s$, thus implying a minimal requirement on growth order of the source term that can prevent blow-up in system \eqref{0.1} when $N\leq 3$. Lastly, with additional monotonicity and convexity properties on $\gamma$, we prove uniform boundedness of classical solutions as well as an explicit upper bound for $u$ in any dimensions $N\geq 1$.
	
	The remaining structure of this paper is as follows. In Section 2, we provide some preliminary results and introduce several useful lemmas that will be utilized in the subsequent proof. Section 3 focuses on derivation of  the uniform boundedness of $v$. In Sections 4 and 5, we prove the uniform boundedness of $u$ in two-dimensional and three-dimensional settings, respectively. Finally, with additional properties indicated in Theorem \ref{thm2}, we derive explicit $L^\infty$-estimates of $u$ in any dimensions in Section 6.
	
	\section{Preliminaries}
	In this section, we will introduce some previously established conclusions that will be utilized later. We start with the existence of local classical solutions to problem \eqref{0.1}, which can be established through the Schauder fixed-point theorem and the regularity theory for elliptic/parabolic equations. Since a quite similar proof can be found in \cite[Lemma 3.1]{AhnYoon2019}, we omit the detail here.
	\begin{theorem} \label{local}
		Let $\Omega \subset \mathbb{R}^N$ with $N\geq1$ be a smooth bounded domain. Assume that  $f(\cdot)$ satisfies condition \eqref{f1}, and $\gamma(\cdot)$ satisfies    \eqref{gamma}. Then for any given initial datum $u^{in}$ satisfying \eqref{uin},  problem \eqref{0.1} has a unique non-negative classical  solution  $(u,v)\in \left( C^0\left( [0,T_{\max} ) \times \bar{\Omega} \right) \cap C^{1,2}((0,T_{\max})\times  \bar{\Omega})   \right)^2,$ defined on a maximal time interval $[0,T_{\max})$ with $T_{\max}\in(0,\infty]$. If  $T_{\max} < \infty,$ then
		$$\lim\limits_{t \rightarrow T_{\mathrm{ma              x}}} \|u(t,\cdot)\|_{L^\infty(\Omega)}=\infty.$$
	\end{theorem}

	\par Next, we recall the following lemma which is provided in \cite{AhnYoon2019,Brezis1973}, regarding estimations for the solution of  Helmholtz equations.  Let $l_+=\max \{ l,0\}$. Then, we have
	\begin{lemma}
		Let $\Omega $ be a smooth bounded domain in $\mathbb{R}^N,N\geq 1,$ and $f\in L^1(\Omega).$ For any $1\leq q< \frac{N}{(N-2)_+},$ there exists a positive constant $C=C(N,q,\Omega)$ such that the solution $z \in W^{1,1}(\Omega)$ to 
		\begin{equation*}\label{0}
			\left\{
			\begin{aligned}
				-\Delta z+z=f,~~~~&x\in \Omega, \\
				\nabla z\cdot \mathbf{n}=0,~~~~& x\in \partial \Omega, 
			\end{aligned}
			\right.
		\end{equation*}
		satisfies
		\begin{equation*}
			\|z\|_{L^q(\Omega)}\leq C(N,q,\Omega)\|f\|_{L^1(\Omega)}.
		\end{equation*}
	\end{lemma}
	
	Then, we introduce the following uniform Gronwall's inequality \cite[Chapter III, Lemma 1.1]{Temam2012}.
	\begin{lemma}\label{Gronwall}
		Let $g,h,y$ be three positive locally integrable functions on $(t_0,+\infty)$ such that $y'$ is locally integrable on $(t_0,+\infty)$ and the following inequalities are satisfied:
		\begin{equation*}
			\frac{\rd y}{\rd t}\leq gy+h, \;\ \forall t\geq t_0,
		\end{equation*}
		\begin{equation*}
			\int_t^{t+r}g(s)\;\rd s\leq a_1,\quad \int_t^{t+r}h(s)\;\rd s \leq a_2,\quad  \int_t^{t+r}y(s)\;\rd s \leq a_3,\;\ \forall t \geq t_0,
		\end{equation*}
		where $r$, $a_i (i=1,2,3)$ are positive constants. Then
		\begin{equation*}
			y(t+r)\leq \left(\frac{a_3}{r}+a_2\right)e^{a_1},\;\ \forall t\geq t_0.
		\end{equation*}
	\end{lemma} 
	
	The following lemma provides lower estimates for $s^\alpha f(s)\log^\beta s$ with any $\alpha>0$ and  $\beta \geq 0$, as well as upper estimates for $sf(s)$, for all $s\geq0$. 
	\begin{lemma}\label{uf(u)}
		Suppose that $f$ satisfies the assumptions \eqref{f1}. Assume that $\alpha>0$ and $\beta\geq0$. Then, for any $a_1>0$, there exists a constant $b_1\geq 0$ depending on $a_1$, $f$, $\alpha$ and $\beta$ such that 
		\begin{equation}\label{28}
			s^\alpha f(s)\log^\beta s\geq a_1s^\alpha-b_1,~~~~\text{for any } s\geq 0.
		\end{equation}
		Furthermore, if $f$ satisfies \eqref{f2}, then there exist constants $a_2 > 0$ and $b_2\geq0$ such that
		\begin{equation}\label{29}
			sf(s)\leq a_2s\log s+b_2,~~~~\text{for any }s\geq 0.
		\end{equation}
	\end{lemma}
	
	\begin{proof}
		Since $f(\cdot)$ satisfies the assumption \eqref{f1}, we can infer that for any $a_1>0$, there exists $s_0>1$ such that for all $s\geq s_0$, $$ s^\alpha f(s)\log ^\beta s \geq a_1 s^\alpha.$$		
		On the other hand, note $$b_1 :=\max\limits_{s\in [0, s_0]} \left|a_1s^\alpha - s^\alpha f(s)\log ^\beta s \right |<\infty,$$ 
		due to the continuity of $f$ and $s^\alpha\log^\beta s$ on $[0,\infty)$.   Then, we have 
		$$s^\alpha f(s)\log ^\beta s \geq a_1s^\alpha-b_1$$ for all $ s\geq0 $.
		
		Similarly, if $f$ satisfies \eqref{f2}, for any $a_2>0$, there exists $s_1>1$, such that for all $s>s_1$,
		\begin{equation*}
			sf(s)\leq a_2 s\log s.
		\end{equation*}
		Furthermore, owing to the continuity of $f$ and $s\log s$ on $[0,\infty)$, we have
		\begin{equation}
			b_2:=\max_{s\in [0,s_0]}\left| sf(s)-a_2s\log s \right |<\infty.
		\end{equation}
		Then, we obtain \eqref{29}. This completes the proof.
	\end{proof}

	\begin{remark}\label{beta}
		If $\alpha=1$  and $\beta=0$, it follows from \eqref{28} that for any $a_1>0$, there exists a constant $b_1\geq 0$ depending on $a_1$ and $f$ such that 
		\begin{equation}\label{46}
			s\leq \frac{sf(s)}{a_1}+\frac{b_1}{a_1},\qquad \text{for any } s\geq 0.
		\end{equation}
		We  denote the corresponding constant $b_1$ in \eqref{46} by $\beta_1$ when $a_1=1$ throughout this paper. Thus, for any $s\geq 0$, there holds
		\begin{equation}\label{44}
			s\leq sf(s)+\beta_1.
		\end{equation}
	\end{remark}
	
	Next, define
	\begin{equation*}
		\begin{aligned}
			&D(\mathcal{A}):=\{~z\in H^2(\Omega)~:~\nabla z\cdot \textbf{n}=0~\text{on}~\partial \Omega~\},\\
			&\mathcal{A}[z]:=z-\Delta z,~~~~z\in D(\mathcal{A}).
		\end{aligned}
	\end{equation*}
	Here, $\Delta $ denotes the usual Laplace operator supplemented with homogeneous Neumann boundary conditions.
	
	We complete this section by introducing the following
	key identity firstly  uncovered in \cite{FuJi2020}, which can be easily obtained by taking $\mathcal{A}^{-1}$ on both sizes  of \eqref{35} and using the fact $v=\mathcal{A}^{-1}[u]$ due to \eqref{36}.
	\begin{lemma}
		The function $v$ satisfies the following key identity
		\begin{equation}\label{k-i}
			v_t+\gamma(v)u=\mathcal{A}^{-1}\left[ \gamma(v)u -uf(u)\right],\qquad\qquad(t,x) \in  (0,T_{\max})\times \Omega.
		\end{equation}
	\end{lemma}

	\section{Time-independent upper bounds for \(v\) }
	
	In this section, we prove the uniform boundedness of $v$ in any dimensions via  comparison approaches.  The main outcome of this section is stated as follows.
	\begin{proposition} \label{V}
		Let  $N\geq 1$. Suppose that $\gamma(\cdot)$ satisfies the assumption \eqref{gamma} and
		$f(\cdot)$ satisfies (\ref{f1}). Then there exists a constant $v^*>0$ depending only on $\gamma$, $f$, $\Omega$ and $u^{in}$ such that for all $t\in [0, T_{\max})$, there holds 
		\begin{equation} \label{rangev}
			\|v(t)\|_{L^\infty(\Omega)} \leq v^*.
		\end{equation}
	\end{proposition}
	\begin{proof}	
		In view of the different asymptotic behavior of $\gamma$ near infinity, we divide our proof of Proposition \ref{V} into two cases  presented separately in Lemma \ref{lemmav} and Lemma \ref{lemmav2} as below.\end{proof}
	\subsection{Uniform upper bound for $v$ with bounded motility}
	To begin with, we consider the case when $\gamma$ is bounded near infinity and we have
	\begin{lemma}\label{lemmav}
		Let $N\geq 1$. Suppose that $\gamma(\cdot)$ satisfies assumption \eqref{gamma} and $\gamma$ is bounded near infinity, i.e., 
		\begin{equation}\label{bounded}
			\limsup_{s\rightarrow \infty} \gamma(s) <\infty.
		\end{equation}
		Assume that	$f(\cdot)$ satisfies (\ref{f1}). Then there exists a  constant $C>0$ depending only on $\gamma$ and $f$ such that for any $t\in [0, T_{\max})$,
		\begin{equation}\label{v}
			\|v(t)\|_{L^\infty(\Omega)} \leq \|v^{in}\|_{L^\infty(\Omega)}+C,
		\end{equation}where $v^{in}:=\mathcal{A}^{-1}[u^{in}].$
	\end{lemma}

	\begin{proof} First, we point out that $0\leq v^{in}\in W^{3,p} (\Omega)$ with any $1\leq p<\infty$ by our assumption \eqref{uin},  regularity theory of elliptic equations and maximum principles.
		
		Due to assumption \eqref{gamma}  and \eqref{bounded}, there exists   a positive constant  $\gamma^*:=\sup_{\tau \in [0,\infty)} \{\gamma(\tau)\}<\infty$, such that  for any $(t,x) \in  [0,T_{\max})\times \Omega$, 
		\begin{equation}\label{39}
			0<  \gamma(v(t,x)) \leq \gamma^*.
		\end{equation} 
		By adding $v$ to both sides of the key identity  \eqref{k-i}, it follows from \eqref{46} and \eqref{39} that
		\begin{equation*}
			\begin{aligned}
				v_t+v+\gamma(v)u+\mathcal{A}^{-1}\left[uf(u)\right]
				&=\mathcal{A}^{-1}\left[\gamma(v)u+u\right]\\
				&\leq \left(\gamma^*+1\right)\mathcal{A}^{-1}[u]\\
				&\leq  \left(\gamma^*+1\right)\mathcal{A}^{-1}\left[ \frac{uf(u)}{a_1}+\frac{b_1}{a_1}\right],
			\end{aligned}
		\end{equation*}
		which  by choosing $a_1=  \gamma^*+1 $, gives rise to
		\begin{equation}\label{t2}
			v_t+v+\gamma(v)u \leq C,
		\end{equation}
		with $C>0$ depending  only on $f$ and $\gamma$. Noticing that $\gamma(v)u$ is non-negative, we can derive \eqref{v} by standard ODE techniques. This completes the proof.
	\end{proof}
	\subsection{Uniform upper bound for $v$ with  unbounded motility}
	When $\gamma$ becomes unbounded near  infinity, the previous  method fails. We then modify  the comparison argument in \cite{JiLa2023} to derive a uniform upper bound for $v$. We begin with the following  auxiliary lemma.

	\begin{lemma}\label{auxiliary}
		Let $N\geq 1$. Suppose that $\gamma$ satisfies assumption \eqref{gamma} and $\gamma$ is unbounded near infinity, i.e.,
		\begin{equation}\label{unbounded}
			\limsup_{s\rightarrow \infty} \gamma(s) =\infty.
		\end{equation}
		Then there exists  a constant $s_* \geq \max \{ \|v^{in}\|_{\infty},\beta_1 \}$  such that
		\begin{equation}\label{s_*}
			\gamma(s_*)=\max_{s\in [
				0,s_*]}\{\gamma(s)\}.
		\end{equation}
		Here,   $\beta_1$ is the  constant specified in Remark \ref{beta}.
	\end{lemma}
	\begin{proof}
		Let $j \geq 1$ be  integers. We define
		\begin{equation*}
			M_j:=\max_{s\in [0,j\|v^{in}\|_{L^\infty(\Omega)}]}\{ \gamma(s)\},
		\end{equation*}
		\begin{equation*}
			s_j:=\sup\{ s\in [0,j\|v^{in}\|_{L^\infty(\Omega)}]: \gamma(s)=M_j\},
		\end{equation*}
		so that 
		\begin{equation*}
			\gamma(s_j)=M_j=\max_{s\in[0,s_j]}\{ \gamma(s)\}.
		\end{equation*}
		Hence, $\{ s_j\}_{j\geq  1}$ and $\{M_j\}_{j\geq 1}$ are non-decreasing sequences, and the unboundedness of $\gamma$ in \eqref{unbounded}  ensures that $\lim_{j\rightarrow \infty} M_j=\infty$ and $\lim_{j\rightarrow \infty} s_j=\infty.$
		It follows that 
		\begin{equation*}
			j_0:=\inf \left\{ j\geq 1,s_j\geq \max\{ \|v^{in}\|_{L^\infty(\Omega)} ,\beta_1\} \right\}<\infty.
		\end{equation*} 
		Let $s_*:= s_{j_0}$. Thus,
		\begin{equation*}
			\gamma(s_*)=\gamma(s_{j_0})=M_{s_{j_0}}=\max_{s\in [0,s_{j_0}]}\{ \gamma(s) \}=\max_{s\in [0,s_*] }\{ \gamma(s)\},
		\end{equation*} and this completes the proof.	\end{proof}

	Next, we set 
	\begin{equation*}
		\gamma'_i(s):=\left\{
		\begin{aligned}
			&0,~~~~&s\in [0,s_*), \\
			&(\gamma'(s))_+=\max\{ \gamma'(s),0\},~~~~& s\geq s_*, 
		\end{aligned}
		\right.
	\end{equation*}
	
	\begin{equation*}\ 
		\gamma'_d(s):=\left\{
		\begin{aligned}
			&0,~~~~&s\in [0,s_*), \\
			&-(\gamma'(s))_-=\min\{ \gamma'(s),0\},~~~~& s\geq s_*, 
		\end{aligned}
		\right.
	\end{equation*}
	with $\gamma_i(s_*)=\gamma_d(s_*)=0$.
	Based on the above definitions, we can deduce the following properties.
	\begin{subequations}\label{gammaid}
		\begin{align}
			&	\gamma_i(s)\geq 0\geq \gamma_d(s)~~~~&s\in [0,\infty),  \label{gammaid1} \\
			&	\gamma(s)=\gamma(s_*)+\gamma_i(s)+\gamma_d(s),~~~~& s\in [s_*,\infty),   \label{gammaid2}\\
			&	\gamma_i(s)=\gamma_d(s)=0,~~~~& s\in [0,s_*).   \label{gammaid3}
		\end{align}
	\end{subequations}
	We also define
	\begin{equation}\label{Gamma}
		\Gamma_d(s)=\int_{s_*}^{s}\gamma_d(\sigma)\;\rd \sigma,~~~~s\in [0,\infty),
	\end{equation}
	and it follows   that 
	\begin{subequations} \label{Gamma2}
		\begin{align}
			&\Gamma_d(s)=0,~~~~&s\in [0,s_*), \label{Gamma21}\\
			&	0\geq \Gamma_d(s)\geq(s-s_*)\gamma_d(s) \geq s\gamma_d(s),~~~~& s\in [s_*,\infty).\label{Gamma22}
		\end{align}
	\end{subequations}
	With these notations, we can derive the following auxiliary conclusion.
	\begin{lemma}
		There holds 
		\begin{equation}\label{ugamma}
			u\gamma(v)\leq u [ \gamma(s_*)+ \gamma_i (\|v\|_{L^\infty(\Omega)}) ]+\mathcal{A}[\Gamma_d(v)]\qquad \text{in}\;\;(0,T_{\max})\times \Omega.
		\end{equation}
	\end{lemma}
	\begin{proof}
		Fix $(t,x)\in (0,T_{\max})\times\Omega$.
		
		If $v(t,x)\geq s_*$,   from  \eqref{36}, \eqref{gammaid1} and \eqref{gammaid2}, as well as the non-negativity of  $u$, we have
		\begin{equation*} 
			\begin{aligned}
				u(t,x)\gamma\left( v(t,x)\right)&=u(t,x)[\gamma(s_*) +\gamma_i(v(t,x))+\gamma_d(v(t,x))]\\
				&\leq u(t,x)\gamma(s_*)+u(t,x)\gamma_i(\|v\|_{L^\infty(\Omega)})+\gamma_d(v(t,x))(v-\Delta v)(t,x)\\
				&= u(t,x)\gamma(s_*)+u(t,x)\gamma_i(\|v\|_{L^\infty(\Omega)})+\gamma_d(v(t,x))v(t,x)\\
				&~~~~-\text{div}(\gamma_d(v)\nabla v)(t,x)+\gamma'_d(v(t,x))|\nabla v(t,x)|^2\\
				&\leq u(t,x)\gamma(s_*)+u(t,x)\gamma_i(\|v\|_{L^\infty(\Omega)})+v(t,x)\gamma_d(v(t,x))-\Delta \Gamma_d(v)(t,x).
			\end{aligned}
		\end{equation*}
		It then follows	from  \eqref{Gamma22} that 
		\begin{equation*} 
			\begin{aligned}
				u(t,x)\gamma\left( v(t,x)\right)&\leq u(t,x)\gamma(s_*)+u(t,x)\gamma_i(\|v\|_{L^\infty(\Omega)})+\Gamma_d(v)(t,x)-\Delta \Gamma_d(v)(t,x)\\
				&= u(t,x)\gamma(s_*)+u(t,x)\gamma_i(\|v\|_{L^\infty(\Omega)})+\mathcal{A}(\Gamma_d(v))(t,x).\\
			\end{aligned}
		\end{equation*}
		
		If $0\leq v(t,x)< s_*$,  the continuity of  $v$ implies that  there exists  $ r>0,$  such that for all $y\in B_r(x)\subset\Omega,$   we have  $0\leq v(t,y) < s_*.$
		Therefore,    $\Gamma_d(v(t)) \equiv 0$ in $ B_r(x)$ and thus $\mathcal{A}[\Gamma_d(v)](t,x)=0.$  Recalling  \eqref{s_*}, we obtain
		$$ u(t,x)\gamma(v(t,x)) \leq u(t,x)\gamma(s_*)\leq u(t,x)[\gamma(s_*)+\gamma_i(\|v\|_{L^\infty(\Omega)})]+\mathcal{A}[\Gamma_d(v)](t,x).$$
		This completes the proof.\end{proof}
	\begin{lemma}\label{lemmav2} 
		Let $N\geq 1$. Suppose that $ \gamma$ satisfies assumptions \eqref{gamma} and \eqref{unbounded}. Then for all $t\in [0, T_{\max})$,
		\begin{equation*}
			\|v(t,\cdot)\|_{L^\infty(\Omega)} \leq s_*.
		\end{equation*}
	\end{lemma}
	\begin{proof}
		Combining \eqref{ugamma} with the elliptic comparison principle yields
		$$ \mathcal{A}^{-1}[u\gamma(v)]\leq [\gamma(s_*)+\gamma_i(\|v\|_\infty)]v+\Gamma_d(v).$$
		Using the key identity \eqref{k-i}, we have
		$$v_t+\gamma(v)u+\mathcal{A}^{-1}[uf(u)]\leq [\gamma(s_*)+\gamma_i(\|v\|_\infty)]v+\Gamma_d(v).$$
		Adding $v$ to both sides, utilizing \eqref{36}  and   \eqref{44}, we obtain
		\begin{equation*} 
			\begin{aligned}
				&v_t-\gamma(v)\Delta v+\gamma(v)v+\mathcal{A}^{-1}[uf(u)]+v\\
				&~~~~~~~~\leq [\gamma(s_*)+\gamma_i(\|v\|_\infty)]v+\Gamma_d(v)+\mathcal{A}^{-1}[u]\\
				&~~~~~~~~\leq [\gamma(s_*)+\gamma_i(\|v\|_\infty)]v+\Gamma_d(v)+\mathcal{A}^{-1}\left[uf(u)+\beta_1\right], 
			\end{aligned}
		\end{equation*}
		which implies
		$$v_t-\gamma(v)\Delta v+\gamma(v)v+	v\leq [\gamma(s_*)+\gamma_i(\|v\|_\infty)]v+\Gamma_d(v)+\beta_1,~~~~ (t,x)\in (0,T_{\max})\times \Omega.$$
		Since $s_*\geq \beta_1$, $v$ satisfies 
		\begin{equation} \label{vunbounded}
			\left\{
			\begin{aligned}
				v_t-\gamma(v)\Delta v+\gamma(v)v+	&v\leq [\gamma(s_*)+\gamma_i(\|v\|_\infty)]v+\Gamma_d(v)+s_*,~~~~&(t,x)\in (0,T_{\max})\times \Omega, \\
				&\nabla v\cdot \textbf{n}=0,~~~~& (t,x)\in (0,T_{\max})\times \partial \Omega,\\
				&v(0)=v^{in},~~~~&x\in \Omega.
			\end{aligned}
			\right.
		\end{equation}
		Next, let $V$ be a solution to the following ordinary differential equation:
		\begin{equation} \label{vtildeunbounded}
			\left\{
			\begin{aligned}
				\frac{ \rd V}{\rd t}+\gamma(V)V+&V= [\gamma(s_*)+\gamma_i(\|v\|_\infty)]V+\Gamma_d(V)+s_*,~~~~&t \in (0,T_{\max}), \\	&V(0)=s_*.
			\end{aligned}
			\right.
		\end{equation}
		Notice that the non-negativity of $\Gamma_d$ ensures that $V$ is well-defined in  $[0,T_{\max})$. Then it follows from \eqref{vunbounded}, \eqref{vtildeunbounded}, the fact $\|v^{in}\|_{L^\infty(\Omega)}\leq s_*$, and the parabolic comparison principle that
		\begin{equation}\label{14}
			\|v(t,\cdot)\|_{L^\infty(\Omega)}\leq V(t),~~~~\forall \;t \in (0,T_{\max}).
		\end{equation}
		On the one hand, for any $T \in (0, T_{\max})$, by the continuity of $v$, the non-negativity of $\gamma$ and $V$,  and the non-positivity of $\Gamma_d$, we infer by Gronwall's inequality that
		\begin{equation*}
			V(t)\leq \mathcal{V}_T:=s_*(1+T)\exp\left\{T\left[ \gamma(s_*)+\sup_{\tau\in[0,T]}\{ \gamma_i(\|v(\tau)\|_\infty)\}\right] \right\},~~~~~t\in[0,T].
		\end{equation*}
		On the other hand, it follows  from \eqref{vtildeunbounded}, \eqref{14} and the monotonicity of $\gamma_i$ that
		\begin{equation*}
			\begin{aligned}
				\frac{ \rd V}{\rd t} +V \leq  [\gamma(s_*)+\gamma_i(V)]V+\Gamma_d(V)-\gamma(V)V+s_*,\qquad t \in (0,T_{\max}).
			\end{aligned} 
		\end{equation*}
		Now setting $G(s):=\Gamma_d(s)-s\gamma_d(s) $ for $s\geq 0$. From \eqref{gammaid3} and \eqref{Gamma21}, we know that  $G(s)=0$ for $s\in [0,s_*]$. It follows that
		\begin{equation*}
			\begin{aligned}
				\frac{\rd (V-s_*)_+}{\;\rd t} +(V -s_*)_+  
				&\leq [\gamma(s_*)+\gamma_i(V)-\gamma (V)]V\text{sign}_+(V-s_*) +\Gamma_d(V)\text{sign}_+(V-s_*)\\
				&= [\gamma(s_*)+\gamma_i(V)-\gamma(s_*) -\gamma_i(V)-\gamma_d(V)]V\text{sign}_+(V-s_*) +\Gamma_d(V)\text{sign}_+(V-s_*)\\
				&= \left[ G(V)-G(s_*) \right] \text{sign}_+(V-s_*)\\
				&= \frac{G(V)-G(s_*) }{V -s_*} (V -s_*)_+.\\
			\end{aligned} 
		\end{equation*}
		Since $G'(s)=-s\gamma'_d(s)\geq 0$ for $s\geq s_*$, we have
		\begin{equation*}
			0\leq \frac{G(V)-G(s_*) }{V -s_*} (V -s_*)_+ \leq (V -s_*)_+ \sup_{[s_*,\mathcal{V}_T] }\{ G'(s)\},~~~~t\in [0,	T].
		\end{equation*}	 
		Thus, we  conclude that
		$ V(t)\leq  s_*,$ for $t\in [0,T]$.
		Since $T$ is arbitrary in $(0,T_{\max})$, the proof is completed by \eqref{14}.
	\end{proof}

	\section{ Time-independent upper bounds for $u$ in 2D}
	In this section, we aim to establish the uniform upper bound for $u$ when $N=2$. To begin with, we recall that  $\gamma(\cdot)$ is continuous and   $0\leq v\leq v^*$ in $[0,T_{\max})\times\bar{\Omega}$ by \eqref{rangev}. Thus, there are two positive constants $\gamma_*$ and $\gamma^*$ such that
	\begin{equation}\label{rangegamma}
		0\leq \gamma_*:= \min_{\tau \in [0,v^*]} \gamma(\tau)\leq \gamma(v(t,x)) \leq \gamma^*:=  \max_{\tau \in [0,v^*]}\gamma(\tau)<\infty,\qquad \text{on}\;\;[0,T_{\max})\times\bar{\Omega}.
	\end{equation}
	
	\begin{lemma}\label{lemmau}
		Assume that $N\geq 1$  and $(u,v)$ represents the classical solution of problem \eqref{0.1} on $[0,T_{\max}) \times \Omega $. Then there exists a constant $C>0$ depending only on $\Omega$, $f$ and $\|u^{in}\|_{L^1(\Omega)}$ such that
		\begin{equation}\label{g}
			\int_\Omega u \;\rd x\leq C, \qquad 	\text{for any } t \in [0,T_{\max}).
		\end{equation}
	\end{lemma}
	\begin{proof}
		Integrating \eqref{35} over $\Omega$,  we obtain that
		$$\frac{\rd}{\rd t}\int_\Omega u \;\rd x=-\int_\Omega uf(u)\;\rd x.$$
		Recall that by \eqref{44},
		$$\int_\Omega u \;\rd x\leq  \int_\Omega uf(u) \;\rd x+\beta_1|\Omega|.$$
		Therefore,
		$$\frac{\rd}{\rd t}\int_\Omega u \;\rd x+\int_\Omega u \;\rd x  \leq  \beta_1|\Omega|.$$  By solving the above differential equation, we complete the proof.
	\end{proof}

	\begin{lemma}
		Assume that $N\geq 1$  and $(u,v)$ represents the classical solution of problem \eqref{0.1} on $[0,T_{\max}) \times \Omega $. Then, there exists a constant $C>0$ depending only on $\gamma$, $f$ and $u^{in}$ such that
		\begin{equation}\label{h}
			v_t+\gamma^*v+\gamma(v)u\leq C,
		\end{equation}
		where $\gamma^*$ is defined in \eqref{rangegamma}.
	\end{lemma}
	\begin{proof}
		We deduce from   \eqref{46}, the key identity \eqref{k-i}, and \eqref{rangegamma}  that
		\begin{equation*}
			\begin{aligned}
				v_t+\gamma(v)u+\gamma^*v&= \mathcal{A}^{-1}[\gamma(v) u- uf(u)+\gamma^*u]\\
				&\leq \mathcal{A}^{-1}[2\gamma^* u- uf(u)]\\
				&\leq \mathcal{A}^{-1}\left[2\gamma^* \left( \frac{uf(u)}{a_1}+\frac{b_1}{a_1}\right)-uf(u)\right]\\
				&=\mathcal{A}^{-1}\left[ \left( \frac{2\gamma^*}{a_1}-1\right)uf(u) \right]+ \frac{2\gamma^*b_1}{a_1}.
			\end{aligned}
		\end{equation*}
		By choosing $a_1=  2\gamma^* $, we have
		\begin{equation*}
			v_t+\gamma(v)u+\gamma^*v \leq C,
		\end{equation*}
		where $C>0$ depending only on $\gamma$, $f$ and $u^{in}$.
		This 	 completes the proof.
	\end{proof}

	\begin{lemma}\label{lemma a}
		Assume that $N\geq 1$  and $(u,v)$ represents the classical solution of problem \eqref{0.1} on $[0,T_{\max}) \times \Omega $.  Then, there exists a constant $C>0$ depending only on $\gamma$, $f$, $\Omega$  and $ u^{in}$  such that for any $t \in [0,T_{\mathrm{max}})$,
		\begin{equation}\label{i}
			\int_\Omega(|\nabla v|^2+v^2)\;\rd x\leq C.
		\end{equation}
		Moreover, for any $t\in (0,T_\mathrm{max}-\tau)$, with $0<\tau<\mathrm{min}\left\{1,\frac{T_\mathrm{max}}{2}\right\}$, we have
		\begin{equation}\label{j}
			\int_{t}^{t+\tau}\int_\Omega u^2 \;\rd x \;\rd s\leq C,
		\end{equation}
		and
		\begin{equation}\label{k}
			\int_{t}^{t+\tau}\int_\Omega  |\nabla v|^4 \;\rd x\;\rd s\leq C.
		\end{equation}
	\end{lemma}
	
	\begin{proof}
		From \eqref{36} and \eqref{h}, there exists a positive constant $C$ depending only on $\gamma$, $f$  and $u^{in}$, such that
		\begin{equation*}
			\begin{aligned}
				\frac{1}{2}\frac{\rd }{\rd t}\int_\Omega \left(|\nabla v|^2+v^2 \right)\;\rd x
				&=\int_\Omega uv_t \;\rd x\\
				&\leq -\int_\Omega u\left(\gamma(v)u+\gamma^*v\right) \;\rd x+C\int_\Omega u \;\rd x.
			\end{aligned}
		\end{equation*}
		Invoking Lemma \ref{lemmau} and using \eqref{36} again, we obtain that 
		\begin{equation*}
			\begin{aligned}
				\frac{1}{2}\frac{\rd}{\rd t}\int_\Omega \left(|\nabla v|^2+v^2 \right)\;\rd x
				&\leq -\int_\Omega\gamma(v)u^2\;\rd x-\int_\Omega \gamma^*uv \;\rd x+C\\
				&=-\int_\Omega\gamma(v)u^2\;\rd x-\gamma^*\int_\Omega \left(|\nabla v|^2+v^2 \right)\;\rd x+C,
			\end{aligned}
		\end{equation*}
		which implies
		$$	\frac{1}{2}\frac{\rd}{\rd t}\int_\Omega \left(|\nabla v|^2+v^2 \right)\;\rd x+\int_\Omega\gamma(v)u^2\;\rd x+\gamma^*\int_\Omega \left(|\nabla v|^2+v^2 \right)\;\rd x\leq C,$$
		where $C>0$ depending only on $\gamma$, $f$, $\Omega$ and $u^{in}$.	We can derive \eqref{i} by standard ODE techniques. Then integrating  the above equality from $t$ to $t+\tau$, and recalling that $\gamma(v)\geq\gamma_*$, we obtain \eqref{j}.
		\par In view of Proposition \ref{V}, an application of the Gagliardo-Nirenberg inequality together with  the regularity theorem for elliptic equations yields that
		\begin{equation}\label{11}
			\| \nabla v\|_{L^4(\Omega)}\leq C\|  v\| _{H^2(\Omega)}^{1/2}\| v\|_{L^\infty(\Omega)}^{1/2}+C\| v\|_{L^\infty(\Omega)} \leq C\| u\|_{L^2(\Omega)}^{1/2}+C
		\end{equation} with a positive constant $C$.
		Therefore, it follows from \eqref{j} that there exists a constant $C>0$ depending on $\gamma$, $f$, $\Omega$ and $ u^{in} $, such that
		\begin{equation*} 
			\int_{t}^{t+\tau}\int_\Omega |\nabla v|^4\;\rd x \;\rd s
			\leq C\int_{t}^{t+\tau}\int_\Omega u^2\;\rd x\,\rd s+C\leq C.
		\end{equation*}This completes the proof.	\end{proof}
	
	\begin{proposition}\label{lemmau2}
		When $N=2$, there exists a constant $C>0$ depending only on $\gamma$, $f$, $\Omega$  and $ u^{in}$  such that 
		\begin{equation}\label{w}
			\|u(t)\|_{L^2(\Omega)}\leq C,\qquad \text{for any } t \in [0,T_{\max}).
		\end{equation}
	\end{proposition}
	
	\begin{proof}
		Multiplying \eqref{35} by $u $ and integrating the resultant over $\Omega$ yields that
		\begin{equation*}
			\frac{1}{2}\frac{\rd}{\rd t}\int_\Omega u^2\;\rd x
			=\int_\Omega u u_t \;\rd x	
			=\int_\Omega u \left(\Delta(\gamma(v)u)-uf(u)\right) \;\rd x,
		\end{equation*}
		where 
		\begin{equation*}
			\begin{aligned}
				\int_\Omega u  \Delta(\gamma(v)u)  \;\rd x
				&=- \int_\Omega \nabla u\cdot \nabla (\gamma(v)u)\;\rd x\\
				&=- \int_\Omega u  \gamma'(v)\nabla u\cdot \nabla v \;\rd x- \int_\Omega \gamma(v)|\nabla u|^2 \;\rd x.
			\end{aligned}
		\end{equation*}
		Recalling that  $v$ has an upper bound $v^*$ and $\gamma$ satisfies the condition  \eqref{gamma}, an application of  Young's inequality yields that
		\begin{equation*}
			\begin{aligned}
				- \int_\Omega u  \gamma'(v)\nabla u\cdot \nabla v \;\rd x
				&\leq \frac{1}{2}\int_\Omega \gamma(v)  |\nabla u|^2\;\rd x+\frac{1}{2}\int_\Omega \frac{|\gamma'(v)|^2}{\gamma(v)}u^2|\nabla v|^2 \;\rd x\\
				&\leq\frac{1}{2}\int_\Omega \gamma(v)  |\nabla u|^2\;\rd x+ K_\gamma \int_\Omega u^2|\nabla v|^2 \;\rd x,
			\end{aligned}
		\end{equation*}
		where $K_\gamma:=\max_{s\in [0,v^*]}  \left\{\frac{|\gamma'(s)|^2}{2\gamma(s)}\right\}$. 
		It follows from \eqref{rangegamma} and the above that
		\begin{equation} \label{9}
			\frac{\rd}{\rd t}\int_\Omega u^2\;\rd x+2\int_\Omega u^2f(u) \;\rd x+ \gamma_*\int_\Omega\left|\nabla u \right|^2 \;\rd x
			\leq 2K_\gamma\int_\Omega u^2|\nabla v|^2 \;\rd x.
		\end{equation}
		By the two-dimensional Gagliardo-Nirenberg inequality and \eqref{g}, we infer that there exists  a constant $C>0$ such that 
		\begin{equation}\label{10}
			\begin{aligned}
				\|u\|_{L^4(\Omega)}&\leq C\left( \|\nabla u\|^{1/2}_{L^2(\Omega)}\|  u\|^{1/2}_{L^2(\Omega)} +\|u\|_{L^1(\Omega)}\right)\\
				&\leq C\left( \|\nabla u\|^{1/2}_{L^2(\Omega)}\|  u\|^{1/2}_{L^2(\Omega)} +1 \right).
			\end{aligned}
		\end{equation}
		Thanks to  H\"older's inequality, \eqref{10}, \eqref{11} and  Young's inequality, we  deduce that for any $\varepsilon>0$, there holds
		\begin{equation*} 
			\begin{aligned}
				\int_\Omega u^2|\nabla v|^2 \;\rd x&\leq  \|u\|_{L^4(\Omega)}^2\|\nabla v\|_{L^4(\Omega)}^2 \\
				&\leq C\left( \|\nabla u\|_{L^2}^{1/2} \|u\|_{L^2(\Omega)}^{1/2} +1 \right)^2\left( \|u\|_{L^2(\Omega)}^{1/2}+1 \right)^2\\
				&\leq C\left( \|\nabla u\|_{L^2(\Omega)} \|u\|_{L^2(\Omega)} +1 \right) \left( \|u\|_{L^2(\Omega)}+1 \right) \\
				&\leq C\left( \|\nabla u\|_{L^2(\Omega)} \|u\|_{L^2(\Omega)}^2+\|\nabla u\|_{L^2(\Omega)} \|u\|_{L^2(\Omega)}+ \|u\|_{L^2(\Omega)}+1 \right)\\
				&\leq \varepsilon \|\nabla u\|_{L^2(\Omega)}^2+C(\varepsilon)  \left( \|u\|_{L^2(\Omega)}^4+1\right).
			\end{aligned}
		\end{equation*}
		Thus,   
		\begin{equation*}\label{12}
			\begin{aligned}
				\frac{\rd}{\rd t}\int_\Omega u^2\;\rd x +2\int_\Omega u^2f(u) \;\rd x+ \gamma_*\int_\Omega\left|\nabla u \right|^2 \;\rd x
				&\leq  2K_\gamma\int_\Omega u^2|\nabla v|^2 \;\rd x\\
				&\leq 2K_\gamma\varepsilon \|\nabla u\|_{L^2(\Omega)}^2+2K_\gamma C(\varepsilon)\left( \|u\|_{L^2(\Omega)}^4+1\right).\\
			\end{aligned}
		\end{equation*}
		Recall  that $u^2f(u)$ has a lower bound $-b_1\leq 0$ corresponding to  parameters $\alpha=2,\beta=0$ and  $a_1$=1 in \eqref{28}. Choosing $2K_\gamma\varepsilon=\gamma_*/2$, we obtain that
		\begin{equation*}
			\frac{\rd}{\rd t}\int_\Omega u^2\;\rd x\leq C\|u\|_{L^2(\Omega)}^4+C=C\|u\|_{L^2(\Omega)}^2\int_\Omega u^2 \;\rd x+C.
		\end{equation*}
		By applying the uniform Gronwall inequality  Lemma \ref{Gronwall} and using \eqref{j}, we conclude that  there exists a constant $C>0$ depending on $\gamma$, $\Omega$, $u^{in}$ and $f$, such that for any $t \in [\tau,T_{\max})$, 
		\begin{equation*}
			\|u\|_{L^2(\Omega)}\leq C
		\end{equation*}and its boundedness on $[0,\tau]$ follows from the local existence result Theorem \ref{local}. This completes the proof.
	\end{proof}
	
	\noindent\textbf{Proof of Theorem \ref{thm1} in 2D.} Once we have established \eqref{w},   we can obtain   by Moser-Alikakos
	iteration  that $ \sup_{t\in [0, T_{\max})} \|u(t)\|_{L^\infty(\Omega)}\leq C$ with a time-independent positive constant $C$ (see, e.g., \cite[Lemma A.1]{TW2012}). Then by Theorem \ref{local}  $T_{\max}=\infty$, and this completes the proof.\qed

	\section{Time-independent upper bounds for $u$ in 3D}
	To get the uniform boundedness of $u$ in 3D, we first derive a  uniform-in-time entropy estimate involving $\int_\Omega u\log u$. Next,  we prove the H\"older continuity of $v$ by establishing a local energy estimate for $v$ following the idea in \cite{LSU1968}. Subsequently, we derive higher-order estimates for $v$ by semigroup theories and finally, we use standard energy method to derive uniform-in-time $L^q$-estimate of $u$ with any $q>\frac{3}{2}$.

	\subsection{Uniform in-time entropy estimate}
	With the help of the uniform-in-time  boundedness of $v$ and  Lemma \ref{lemmau} and Lemma \ref{lemma a},  we can  derive following  entropy estimate.
	\begin{lemma}\label{ulogu}
		Assume that $N\geq 1$  and $(u,v)$ represents the classical solution of problem \eqref{0.1} on $[0,T_{\max}) \times \Omega $. Then, there exists a constant $C>0$  depending on   $\gamma$, $f$, $\Omega$ and $u^{in}$  such that for any $t \in [0,T_{\max})$, there holds
		\begin{equation}\label{n}
			\int_\Omega u\log u \;\rd x\leq C.
		\end{equation}
		
	\end{lemma}
	\begin{proof}
		Multiplying  \eqref{35} by $\log u$ and integrating over $\Omega$, we obtain that
		\begin{equation*}\label{o}
			\frac{\rd}{\rd t}\left(\int_\Omega u\log u \;\rd x-\int_\Omega u \;\rd x\right)=\int_\Omega \Delta(\gamma(v)u)\log u\;\rd x-\int_\Omega u f(u)\log u \;\rd x.
		\end{equation*}
		By integrating by parts and employing Young's inequality, we infer that
		\begin{equation*}
			\begin{aligned}
				\int_\Omega \Delta(\gamma(v)u)\log u \;\rd x
				&=-\int_\Omega \nabla\left(\gamma(v)u\right)\cdot \nabla \log u \;\rd x\\
				&=-\int_\Omega\gamma'(v)\nabla u \cdot \nabla v\;\rd x-\int_\Omega \gamma(v)\frac{|\nabla u|^2}{u}\;\rd x\\
				&\leq\int_\Omega  \frac{|\gamma'(v)|^2}{\gamma(v)} u |\nabla v|^2  \;\rd x+ \int_\Omega  \gamma(v)\frac{|\nabla  u|^2}{u} \;\rd x -\int_\Omega \gamma(v)\frac{|\nabla u|^2}{u}\;\rd x.
			\end{aligned}
		\end{equation*}
		Therefore,  
		\begin{equation*}
			\frac{\rd}{\rd t} \left(\int_\Omega u\log u \;\rd x-\int_\Omega u \;\rd x\right)+\int_\Omega u f(u)\log u \;\rd x  \leq \int_\Omega  \frac{|\gamma'(v)|^2}{\gamma(v)} u |\nabla v|^2  \;\rd x  \leq C\int_\Omega |\nabla v|^4  \;\rd x+ C\int_\Omega u^2 \;\rd x,
		\end{equation*}
		where $C>0$ depending on $\gamma$, $f$, $u^{in}$ and $\Omega$.
		Taking $\alpha=1,\beta=1$ and $a_1=1$ in \eqref{28}, we  know that $u f(u)\log u$ has lower bound $-b_1\leq 0$. Therefore, we have 
		\begin{equation*}\label{s}
			\frac{\rd}{\rd t}\left(\int_\Omega u\log u \;\rd x-\int_\Omega u \;\rd x\right)  \leq C\int_\Omega |\nabla v|^4 \;\rd x + C\int_\Omega u^2\;\rd x+b_1|\Omega|.
		\end{equation*}
		Notice that
		$$-e\leq s\log s-s \leq  s^2,\qquad \text{for any } s \geq 0.$$
		With the aid of Lemma \ref{lemma a}, we know that for any $t\in (0,T_{\max}-\tau)$ with  $0<\tau<\min\{1,\frac{T_{\max}}{2}\}$, there holds
		\begin{equation}\label{45}
			\int_t^{t+\tau}\int_\Omega (u\log u-u)\;\rd x\;\rd s\leq \int_t^{t+\tau}\int_\Omega u^2 \;\rd x\;\rd s\leq C
		\end{equation}
		with $C>0$ depending on $\gamma$, $f$, $\Omega$ and $u^{in}$.  Invoking the uniform Gronwall inequality Lemma \ref{Gronwall}, \eqref{k}, Lemma \ref{lemmau}, Lemma \ref{lemma a} together with the local existence result Theorem \ref{local},  we arrive at 
		$$\int_\Omega u\log u \;\rd x \leq C,~~~~~~\text{~for~any~~} t\in [0,T_{\max}) $$
		with $C>0$ depending on $\gamma$, $f$, $\Omega$ and $u^{in}$. This completes the proof.
	\end{proof}	
	\begin{remark}\label{uf}
		Let $N\geq 1$. Assume that $f(\cdot)$ satisfies   \eqref{f1} and \eqref{f2}. Then, there exists a constant $C>0$ depending on $\gamma$, $f$, $\Omega$ and $u^{in}$ such that
		\begin{equation*} 
			\int_\Omega u f(u) \;\rd x\leq C, \;\;\text{for any}\;\; t\in [0,T_{\max}).
		\end{equation*}
	\end{remark}
	
	\begin{proof}
		It follows from \eqref{29} and Lemma \ref{ulogu} that
		\begin{equation*}
			\int_\Omega uf(u) \;\rd x
			\leq \int_\Omega \left( a_2 u\log u+b_2\right)\;\rd x
			= a_2  \int_\Omega u\log u \;\rd x+b_2|\Omega| 
			\leq C,
		\end{equation*}
		where   $C>0$ depending on $\gamma$, $f$, $\Omega$ and $u^{in}$.
	\end{proof}
	
	\subsection{H\"older Continuity of $v$}
	Based on the $L^1$-estimate of the nonlinear term $uf(u)$ in Remark \ref{uf}, we  prove the H\"older continuity of $v$ by establishing a local energy estimate as done in \cite{JiLa2021}.	In the subsequent parts, we fix  $T \in (0,T_{\max})$, set $J_T=[0,T]$, and let \begin{equation}\label{phi}
		\varphi: =\mathcal{A}^{-1}[u\gamma(v)-uf(u)].
	\end{equation} 
	\begin{lemma}\label{vholder}
		Let $\delta \in (0,1)$. There exists $ C>0$  depending only on $\gamma$, $f$, $u^{in}$ and $\Omega$ such that, if $\vartheta \in C^\infty (J_T\times \bar{\Omega})$, $ 0\leq \vartheta \leq 1,~\sigma\in \left\{-1,1 \right\}, $ and $h\in \mathbb{R}$ are such that  
		\begin{equation}\label{1}
			\sigma v(t,x)-h\leq \delta,~~~~(t,x)\in \mathrm{supp}~ \vartheta,
		\end{equation} 
		then 
		\begin{multline*}
			\int_\Omega \vartheta^2(\sigma v(t)-h)_+^2 \;\rd x+\frac{\gamma_*}{2}\int_{t_0}^t \int _\Omega \vartheta ^2|\nabla (\sigma v(\tau) -h)_+)|^2 \;\rd x \;\rd \tau \\
			\leq \int_\Omega \vartheta ^2(\sigma v(t_0)-h)_+^2\;\rd x+C\int _{t_0}^t \int _\Omega \left( |\nabla \vartheta |^2+\vartheta |\vartheta _t|\right)(\sigma v(\tau )-h)_+^2 \;\rd x \;\rd \tau+C\int _{t_0}^t \left(\int _{A_{h,\sigma}(\tau)} \vartheta \;\rd x \right)^{\frac{1}{2}}\;\rd \tau
		\end{multline*}
		for $0\leq t_0 \leq t\leq T$, 
		where
		$$A_{h,\sigma}(\tau)\triangleq\left\{ x\in \Omega :~\sigma v(\tau ,x)>h\right\},~~~~\tau \in [0,T].$$
	\end{lemma}
	\begin{proof}
		Noticing that $v$ satisfies the key identity \eqref{k-i}, we have
		\begin{equation*}
			\begin{aligned}
				\frac{1}{2}\frac{\rd }{\rd t}\int_\Omega \vartheta^2 (\sigma v-h)_{+}^2\;\rd x
				&= \sigma \int_\Omega \vartheta ^2(\sigma v-h)_+v_t \;\rd x+\int_\Omega  \vartheta \vartheta_t(\sigma v-h)^2_+\;\rd x\\
				&=-\sigma \int_\Omega \vartheta ^2(\sigma v-h)_+u\gamma(v) \;\rd x+\sigma \int_\Omega \vartheta ^2(\sigma v-h)_+\varphi \;\rd x+\int_\Omega  \vartheta \vartheta_t(\sigma v-h)^2_+\;\rd x,
			\end{aligned}		
		\end{equation*}
		where $\varphi$ is defined in \eqref{phi}.
		Either $\sigma =1$ and it follows from \eqref{36}, \eqref{gamma} and the non-negativity of $u\gamma(v)$ and $v$ that
		\begin{equation*}
			\begin{aligned}
				-\sigma \int_\Omega \vartheta ^2(\sigma v-h)_+u\gamma(v) \;\rd x
				&\leq 	-\gamma_* \int_\Omega \vartheta ^2( v-h)_+u \;\rd x\\
				&= -\gamma_*\int_\Omega \vartheta ^2( v-h)_+(v-\Delta v) \;\rd x \\
				&\leq -\gamma_*\int_\Omega \nabla \left(\vartheta ^2(  v-h)_+ \right)\cdot \nabla v \;\rd x\\
				&\leq -\gamma_*\int_\Omega \vartheta^2 |\nabla(v-h)_+|^2 \;\rd x +2\gamma_*\int _\Omega \vartheta|\nabla \vartheta|(v-h)_+|\nabla (v-h)_+| \;\rd x.
			\end{aligned}
		\end{equation*}
		Or $\sigma =-1$ and we infer from \eqref{36}, \eqref{gamma} and \eqref{rangev} that
		\begin{equation*}
			\begin{aligned}
				-\sigma \int_\Omega \vartheta ^2(\sigma v-h)_+u\gamma(v) \;\rd x
				&\leq\gamma^*\int_\Omega \vartheta ^2(-v-h)_+u \;\rd x\\
				&= \gamma^*\int_\Omega \vartheta ^2(-v-h)_+(v-\Delta v) \;\rd x\\
				&\leq \gamma^* v^*\int_\Omega \vartheta^2(-v-h)_+\;\rd x+\gamma^* \int_\Omega \nabla \left(\vartheta ^2(- v-h)_+ \right)\cdot \nabla v \;\rd x\\
				&\leq \gamma^* v^*\int_\Omega \vartheta ^2(-v-h)_+ \;\rd x+2\gamma^*\int _\Omega \vartheta|\nabla \vartheta|(-v-h)_+ |\nabla (-v-h)_+|\;\rd x\\
				&~~~~~-\gamma^*\int_\Omega \vartheta^2 |\nabla(-v-h)_+|^2\;\rd x.\\
			\end{aligned}
		\end{equation*}
		Since $\gamma_* \leq \gamma^*$,  one obtains the following estimate
		\begin{equation*}
			\begin{aligned}
				-\sigma \int_\Omega \vartheta ^2(\sigma v-h)_+u\gamma(v) \;\rd x
				&\leq -\gamma_*\int_\Omega  \vartheta^2 |\nabla(\sigma v-h)_+|^2 \;\rd x+2\gamma^*\int _\Omega \vartheta|\nabla \vartheta| (\sigma v-h)_+ |\nabla (\sigma v-h)_+|  \;\rd x\\
				&~~~~+\gamma^*v^*\int_\Omega \vartheta ^2(\sigma v-h)_+  \;\rd x.
			\end{aligned}
		\end{equation*}
		Invoking  Young's inequality, 
		\begin{equation*}
			\begin{aligned}
				\int _\Omega \vartheta|\nabla \vartheta| (\sigma v-h)_+ |\nabla (\sigma v-h)_+|  \;\rd x
				&\leq \frac{1}{4}\int_\Omega \vartheta^2 |\nabla (\sigma v-h)_+|^2 \;\rd x+ \int_\Omega |\nabla \vartheta |^2(\sigma v-h)^2_+\;\rd x.
			\end{aligned}
		\end{equation*}
		Therefore, we arrive at
		\begin{equation*}
			\begin{aligned}
				\frac{1}{2}\frac{\rd }{\rd t}\int_\Omega \vartheta^2 (\sigma v-h)_{+}^2\;\rd x
				&\leq -\gamma_*\int_\Omega  \vartheta^2 |\nabla(-v-h)_+|^2 \;\rd x+\frac{\gamma_*}{2}\int_\Omega \vartheta^2 |\nabla (\sigma v-h)_+|^2 \;\rd x\\
				&\qquad+2\gamma^*\int_\Omega |\nabla \vartheta |^2(\sigma v-h)^2_+\;\rd x +\gamma^*v^*\int_\Omega \vartheta ^2(\sigma v-h)_+  \;\rd x\\
				&\qquad +\sigma \int_\Omega \vartheta ^2(\sigma v-h)_+\varphi \;\rd x+\int_\Omega  \vartheta \vartheta_t(\sigma v-h)^2_+\;\rd x\\
				&\leq -\frac{\gamma_*}{2}\int_\Omega \vartheta^2 |\nabla (\sigma v-h)_+|^2 \;\rd x+C\int_\Omega |\nabla \vartheta |^2(\sigma v-h)^2_+\;\rd x \\ &\qquad+C\int_\Omega \vartheta ^2(\sigma v-h)_+  \;\rd x
				+\int_\Omega \vartheta ^2(\sigma v-h)_+|\varphi| \;\rd x+\int_\Omega  \vartheta \vartheta_t(\sigma v-h)^2_+\;\rd x
			\end{aligned} 
		\end{equation*} 
		with $C>0$ depending on $f$, $\Omega$, $\gamma$ and $u^{in}$,	which implies
		\begin{equation}\label{2}
			\begin{aligned}
				\frac{1}{2}\frac{\rd }{\rd t}&\int_\Omega \vartheta^2 (\sigma v-h)_{+}^2\;\rd x+\frac{\gamma_*}{2}\int_\Omega \vartheta^2 |\nabla (\sigma v-h)_+|^2 \;\rd x\\
				&\leq C\int_\Omega(|\nabla \vartheta |^2+\vartheta |\vartheta_t|) (\sigma v-h)^2_+\;\rd x  
				+C\int_\Omega \vartheta ^2(\sigma v-h)_+(|\varphi| +1)\;\rd x.
			\end{aligned}
		\end{equation}
		Since
		\begin{equation*}
			\begin{aligned}
				\int_\Omega \vartheta ^2(\sigma v-h)_+(|\varphi| +1)\;\rd x\leq  \|\vartheta^2(\sigma v-h)_+\|_{L^2(\Omega)} \| |\varphi| +1 \|_{L^2(\Omega)},
			\end{aligned}
		\end{equation*}
		and recalling that $0<\delta <1$ and  $0 \leq \vartheta \leq 1$, we infer from  \eqref{1} that
		\begin{equation*} 
			\|\vartheta^2(\sigma v-h)_+\|_{L^2(\Omega)}
			=	\left( \int_\Omega \vartheta ^4 (\sigma v-h)_+^2 \;\rd x\right)^{\frac{1}{2}} 
			\leq \delta \left( \int_{A_{h,\vartheta} (\tau)} \vartheta ^4\;\rd x \right)^\frac{1}{2} 
			\leq   \left( \int_{A_{h,\vartheta }(\tau)} \vartheta \;\rd x \right)^\frac{1}{2}. 
		\end{equation*}
		On the other hand,   we deduce from \eqref{phi},  Lemma \ref{0} and Lemma \ref{ulogu} that
		\begin{equation*}
			\begin{aligned}
				\| |\varphi | +1\|_{L^2(\Omega)}
				&\leq \|\mathcal{A}^{-1}\left[u\gamma(v) \right]+1\|_{L^2(\Omega)}+\|\mathcal{A}^{-1}[uf(u)]\|_{L^2{(\Omega})}\\
				&\leq (\gamma^* v^*+1)|\Omega|^\frac{1}{2}+C\|uf(u)\|_{L^1(\Omega)}\\
				&\leq C_3.
			\end{aligned}
		\end{equation*}
		Here, $C_3>0$ depending on $\Omega$, $\gamma$, $u^{in}$ and $f$. Therefore, 
		\begin{equation*} 
			\begin{aligned}
				\int_\Omega \vartheta ^2(\sigma v-h)_+(|\varphi| +1)\;\rd x\leq C\left( \int_{A_{h,\vartheta}(\tau)} \vartheta \;\rd x \right)^\frac{1}{2}.\\
			\end{aligned}
		\end{equation*}
		Inserting the above estimate in \eqref{2} and integrating from $t_0$ to $t$ completes the proof.
	\end{proof}
	
	We are now in a position to apply \cite[Theorem 8.2]{LSU1968} to obtain a uniform-in-time H\"older estimate for $v$. 
	\begin{proposition}\label{p1}
		When $N=3$, there exists a constant $\alpha \in (0,1)$ depending only on $f$, $\Omega$, $\gamma$ and $u^{in}$  such that $v \in C^\alpha(J_T\times  \bar{\Omega}).$
	\end{proposition}
	\begin{proof}
		According to Lemma \ref{vholder}, the estimate in \cite [Chapter II, Equation (7.5)] {LSU1968} holds true with the parameters $q=\frac{14}{3},r=\frac{7}{3},\kappa=\frac{1}{6}$ satisfying \cite [Equation (7.3)]{LSU1968}.  Therefore, by \cite[Remark 7.2, Remark 8.1]{LSU1968}, we  conclude that  there exists a positive constant $C$ depending only on $\Omega$, $\gamma$,  and $u^{in}$  such that $v$ belongs to the function class  $\hat{\mathcal{B}}_2\left( J_T \times \bar{\Omega},v^*,C,\frac{7}{3},\delta,\frac{1}{6}\right)$  defined in \cite[Chapter II Section 8]{LSU1968}. Taking into account the smoothness of the boundary of  $\Omega$ and the H\"older continuity of $v^{in}\in C^{\alpha_0} (\bar{\Omega})$, with some	$\alpha_0 \in (0,1)$, we can  infer by \cite[Lemma 8.1 \& Theorem 8.2]{LSU1968}  that $v \in C^\alpha (J_T \times \bar{\Omega})$, where $\alpha \in (0,1)$ depending on  $f$, $\Omega$, $\gamma$ and $u^{in}$. Moreover, since the local energy estimate is time-independent, we conclude that the H\"older estimate of $v$ is uniform-in-time as well.	\end{proof}
	
	\subsection{Higher-order estimates of $v$}
	In this part, we will employ the theory of evolution operators to get  higher-order estimates of $v$.
	\begin{lemma}\label{vgaojie}
		When $N=3$, for any $p\in (1,3)$ and $\theta \in (\frac{1+p}{2p},1)$, there exists a constant $C>0$  depending on $\gamma$, $f$, $\Omega$, $u^{in}$, $p$ and $\theta$  such that 
		$$\|v(t)\|_{W^{2\theta,p}(\Omega)}\leq C,~~~~t \in [0,T_{\max}).$$
	\end{lemma}
	\begin{proof}
		Let $T \in (0,T_{\max})$. Denote $J_T=[0,T]$ and $D_T\triangleq(-1,2v^*)$, where  $v^*$ is defined in \eqref{rangev}.   For $s\in D_T$, we define the elliptic operator $\textsf{A} (\cdot)z \triangleq -\gamma(s)\Delta z$, and the boundary operator $\textsf{B}z\triangleq \nabla z \cdot \textbf{n}.$  According to the assumption \eqref{gamma} on 	$\gamma$, we know that $\min_{D_T}\left\{ \gamma\right\}>0$. Thus, for any $s \in D_T$, the operator $\textsf{A}(\cdot)$ satisfies    \cite[Equation (4.6)]{Aman1990}, making it a strongly uniformly elliptic operator. By    \cite[Theorem 4.2]{Aman1990}, it can be concluded that for any $ s \in D_T,$  $(\textsf{A}(\cdot),\textsf{B})$ is normally elliptic. For $t\in J_T,$ 
		let $\tilde{\textsf{A}}(t)$  be the $L^p$ restriction of $(\textsf{A}(\cdot),\textsf{B})$ with domain 
		\begin{equation*}
			W_{\textsf{B}}^{2,p}(\Omega)\triangleq \left\{
			z\in W^{2,p}(\Omega):~~\nabla z \cdot \textbf{n}=0 ~~on~~ \partial \Omega \right\}.
		\end{equation*}
		As in the proof of   \cite[Theorem 6.1]{Aman1989}, the assumption on $\gamma$ and Proposition \ref{p1} ensure that
		\begin{equation*}
			\tilde{\textsf{A}}\in BUC^{\alpha}(J_T,\mathcal{L}(W^{2,p}(\Omega),L^p(\Omega))),
		\end{equation*}
		and   $\tilde{\textsf{A}}(J_T)$ is a regular bounded subset of $C^\alpha(J_T,\mathcal{ H}(W^{2,p}(\Omega),L^p(\Omega)))$ in the sense of  \cite[Section 4]{Aman1988}.  According to   \cite[Theorem A.1]{Aman1990}, there exists a unique parabolic fundamental solution $\tilde{U}$ corresponding to $\tilde{\textsf{A}}(t)$, as well as  time-independent positive constants $M$ and $\omega$, such that for $0\leq \tau \leq t \leq T$, it holds that
		\begin{equation}\label{4}
			\|\tilde{U}(t,\tau)\|_{\mathcal{L}(W^{2,p}(\Omega))}+\|\tilde{U}(t,\tau)\|_{\mathcal{L}(L^p(\Omega))}+(t-\tau)\|\tilde{U}(t,\tau)\|_{\mathcal{L}(L^p(\Omega),W^{2,p}(\Omega))}\leq Me^{\omega}(t-\tau).
		\end{equation}
		Since $\theta \in (\frac{1+p}{2p},1)$, according to    \cite[Theorem 5.2]{Aman1993}, it follows that
		\begin{equation*}
			\left( L^p(\Omega),W_\textsf{B}^{2,p}(\Omega) \right)_{\theta,p} \dot{=} W_\textsf{B}^{2\theta,p}(\Omega)\triangleq \left\{z\in W^{2\theta ,p}(\Omega):~~\nabla z\cdot \textbf{n}=0 ~~on~~ \partial  \Omega
			\right\},
		\end{equation*}
		and according to   \cite[Lemma II.5.1.3]{Aman1995}, there exists a  time-independent constant $M_\theta>0$ such that 
		\begin{equation}\label{5}
			\|\tilde{U}(t,\tau)\|_{\mathcal{L}(W^{2,p}_{\mathcal{B}}(\Omega))}+(t-\tau)^\theta \|\tilde{U}(t,\tau)\|_{\mathcal{L}(L^p(\Omega),W^{2,p}_{\mathcal{B}}(\Omega))}\leq M_\theta e^{\omega(t-\tau)},
		\end{equation}
		for $0\leq \tau \leq t \leq T  $.
		\par Choose $\mu>\omega$, From  \eqref{0.1}, we may infer that $v$ is a solution of the following equation
		\begin{equation*}
			\left\{
			\begin{aligned}
				& v_t+(\mu+\tilde{\textsf{A}}(\cdot ))v=F, \qquad &t \in J_T, \\
				&  v(0)=v^{in},
			\end{aligned}
			\right.
		\end{equation*}
		where $$F(t,x)\triangleq \varphi -v\gamma(v)+\mu v=\mathcal{A}^{-1}[u\gamma(v)-uf(u)]-v\gamma(v)+\mu v,\qquad (t,x) \in (J_T\times \Omega).$$ Considering the boundedness of $v$ and the continuity of $\gamma$, we infer by Lemma \ref{0} and Remark \ref{uf} that
		\begin{equation}\label{6}
			\|F\|_{L^p(\Omega)}\leq C,\qquad t\in J_T,
		\end{equation}
		for any $1<p<3$, where $C$ is a positive constant independent of $t$ and $T$.
		Applying  \cite[Remark II.2.1.2(a)]{Aman1995}, it can be concluded that $v$ has the expression
		\begin{equation}\label{7}
			v(t)=e^{-\mu t}\tilde{U}(t,0)v^{in}+\int_{0}^{t} e^{-\mu(t-\tau )} \tilde{U}(t,\tau)F(\tau)\;\rd \tau,~~t \in J_T.
		\end{equation}
		For $t \in J_T$, \eqref{5} and \eqref{6}  together with \eqref{7} give rise to the following estimate
		\begin{equation*}
			\begin{aligned}
				\|v(t)\|_{W^{2\theta,p}(\Omega)}
				&\leq M_\theta e^{(\omega-\theta)t }\|v^{in}\|_{W^{2\theta,p}}+M_\theta\int_0^t(t-\tau)^{-\theta}e^{(\omega-\mu)(t-\tau)}\|F(\tau)\|_{L^p(\Omega)}\;\rd \tau\\
				&\leq C(p,\theta) \|u^{in}\|_{L^p(\Omega)}+CM_\theta \int_0^t(t-\tau)^{-\theta}e^{(\omega-\mu)(t-\tau)}\;\rd \tau\\
				&\leq C(p,\theta),
			\end{aligned}
		\end{equation*}
		since 
		\begin{equation*}
			\mathcal{I}_\theta \triangleq \int_{0}^{\infty} \tau^{-\theta}e^{(\omega-\mu)\tau}\;\rd \tau <\infty.
		\end{equation*}
		Since $T$ is arbitrary, it follows that 
		$	\|v(t)\|_{W^{2\theta,p}(\Omega)}\leq C$  for any $t \in [0,T_{\max})$. This completes the proof.
	\end{proof}
	\begin{corollary}\label{88}
		When $N=3$,  there exists a constant $C>0$ which    depends on $\gamma$, $f$, $\Omega$ and $u^{in}$ such that 
		\begin{equation}\label{8}
			\|\nabla v (t)\|_{L^4(\Omega)}\leq C,~~~~~t \in [0,T_{\max}).
		\end{equation}
	\end{corollary}
	\begin{proof}Since $2\theta -1\in (\frac{1}{p},1)$, we have
		$(2\theta-1)p\in \left(1,p\right)$, with any $p\in(1,3)$. It then follows from the 3D Sobolev embedding theorem (see, e.g., \cite[Eqn. (5.9)]{Aman1993}) and Lemma \ref{vgaojie} that $\|\nabla v\|_{L^4(\Omega)}\leq C$ for $t\in[0,T_{\max})$.
	\end{proof}

	\subsection{Uniform-in-time boundedness for $\|u\|_{L^q(\Omega)}$ with $q>3/2$}
	
	\begin{lemma}
		When $N=3$, for any $q\in \left(\frac{3}{2},\infty \right)$, there exists a constant $C>0$ depending on $\gamma$, $f$, $\Omega$, $u^{in}$ and $q$  such that 
		\begin{equation}\label{u}
			\|u(t)\|_{L^q(\Omega)}\leq C,~~~~~~~t \in [0,T_{\max}).
		\end{equation}
	\end{lemma}
	
	\begin{proof}
		Multiplying \eqref{35} by $u^{q-1}$  and integrating the resultant over $\Omega$ yields that
		\begin{equation}\label{40}
			\begin{aligned}
				\frac{1}{q}\frac{\rd}{\rd t}\int_\Omega u^q\;\rd x
				&=\int_\Omega u^{q-1}u_t \;\rd x	\\
				&=\int_\Omega u^{q-1}(\Delta(\gamma(v)u)-uf(u)) \;\rd x,
			\end{aligned}
		\end{equation}
		where 
		\begin{equation*}
			\begin{aligned}
				\int_\Omega u^{q-1} \Delta\left(\gamma(v)u\right)   \;\rd x
				&=-(q-1)\int_\Omega u^{q-2}\nabla u\cdot \nabla (\gamma(v)u)\;\rd x\\
				&=-(q-1)\int_\Omega \gamma'(v) u^{q-1}\nabla u\cdot \nabla v \;\rd x-(q-1)\int_\Omega \gamma(v)u^{q-2}|\nabla u|^2 \;\rd x.
			\end{aligned}
		\end{equation*}
		Recalling that  $0\leq v\leq v^*$, and $\gamma$ satisfies the condition  \eqref{gamma},  we apply  Young's inequality to obtain that 
		\begin{equation*}
			\begin{aligned}
				-(q-1)\int_\Omega \gamma'(v)u^{q-1} \nabla u\cdot \nabla v \;\rd x
				&\leq \frac{q-1}{2}\int_\Omega \gamma(v) u^{q-2}|\nabla u|^2\;\rd x+\frac{q-1}{2}\int_\Omega \frac{|\gamma'(v)|^2}{\gamma(v)}u^q|\nabla v|^2 \;\rd x\\
				&\leq\frac{q-1}{2}\int_\Omega \gamma(v) u^{q-2}|\nabla u|^2\;\rd x+ C\int_\Omega u^q|\nabla v|^2 \;\rd x,
			\end{aligned}
		\end{equation*}with a constant $C>0$ depending on $q$, $\gamma$, $f$ and $u^{in}$.
		Since  $u^{q-2}|\nabla u|^2= \frac{4}{q^2} \left|\nabla u^{\frac{q}{2}}\right|^2$, it follows from  \eqref{40} and the above estimate that
		\begin{equation}\label{x}
			\frac{\rd}{\rd t}\int_\Omega u^q\;\rd x+q\int_\Omega u^qf(u) \;\rd x+C_3 \int_\Omega\left|\nabla u^{\frac{q}{2}}\right|^2 \;\rd x
			\leq C\int_\Omega u^q|\nabla v|^2 \;\rd x,
		\end{equation}
		where $C_3$ is a positive constant and $C_3:=\min_{s\in[0,v^*]}\{\frac{2(q-1)}{q}\gamma(s) \}$.	Observing that $2s^q\leq {qs^qf(s)} +{qb_1}$ for all $s\geq0$ by taking $a_1=\frac{2}{q}$, $\alpha =q$, and $\beta=0$ in \eqref{28}, we then obtain that
		\begin{equation}\label{47}
			\frac{\rd}{\rd t}\int_\Omega u^q\;\rd x+C_3\int_\Omega\left|\nabla u^{\frac{q}{2}}\right|^2 \;\rd x+2\int_\Omega u^q \;\rd x\leq C\int_\Omega u^q|\nabla v|^2 \;\rd x +C
		\end{equation}
		with $C>0$ depending on $q$, $f$, $u^{in}$ and $\Omega$.
		Thanks to Corollary \ref{88}, we  use   H\"older's inequality to deduce that
		\begin{equation}\label{48}
			\int_\Omega u^q|\nabla v|^2\;\rd x\leq \left(\int_\Omega|\nabla v|^4\;\rd x \right)^{\frac{1}{2}}\left(\int_\Omega u^{2q} \;\rd x \right)^{\frac{1}{2}}\leq C\|u\|_{L^{2q}(\Omega)}^q. 
		\end{equation}
		Furthermore, thanks to \eqref{g}, by interpolation  and Young's inequality, there exists $\zeta=\frac{6q-3}{6q-2} \in (0,1)$ such that  for any $\varepsilon>0$, 
		\begin{equation} \label{49}
			\|u\|_{L^{2q}(\Omega)}^q
			\leq \|u\|_{L^1(\Omega)}^{q(1-\zeta)}\|u\|_{L^{3q}(\Omega)} ^{q\zeta} 
			\leq C\|u\|_{L^{3q}(\Omega)}^{q\zeta} 
			\leq \varepsilon \|u\|_{L^{3q}(\Omega)}^q+C(\varepsilon).
		\end{equation}
		On the other hand, by the 3D Sobolev embedding $L^6(\Omega)\hookrightarrow H^1(\Omega)$, we infer that
		\begin{equation} \label{50}
			\|u\|_{L^{3q}(\Omega)}^q=\|u^{\frac{q}{2}}\|_{L^{6}(\Omega)}^2\leq C\left(	\int_\Omega u^q \;\rd x+\int_\Omega |\nabla u^{\frac{q}{2}}|^2 \;\rd x\right) 
		\end{equation}
		Thus, picking $\varepsilon$ sufficiently small in \eqref{49}, it follows from \eqref{47}-\eqref{50} that
		\begin{equation*}
			\frac{\rd}{\rd t}\int_\Omega u^q\;\rd x+ \int_\Omega u^q \;\rd x +\int_\Omega |\nabla u^{\frac{q}{2}}|^2 \;\rd x
			\leq C,
		\end{equation*}
		where $C>0$ depending on $\gamma$, $f$, $\Omega$, $q$, and $u^{in}$. Then solving the above ODI completes the proof.
	\end{proof}

\noindent	\textbf{Proof of Theorem \ref{thm1} in 3D.}  Once we have \eqref{u}, we can derive the $L^\infty(\Omega)$-boundedness of $u$ by standard iterations  (see, e.g., \cite[Lemma 4.3]{AhnYoon2019}).
Then it follows from Theorem \ref{local} that $T_{\max} = \infty$, and we  complete the proof.

\section{Proof of Theorem \ref{thm2}}
This section is devoted to the proof of Theorem \ref{thm2}. We recall that $\gamma$ is non-decreasing and concave. 
\begin{lemma}
	Under the assumption of Theorem \ref{thm2}, for any $t\in [0,T_{\max})$, there holds
	\begin{equation}\label{41}
		\|u(t,\cdot )\|_{L^\infty(\Omega)}\leq \max \{ \|u^{in}\|_{L^\infty(\Omega)}, \,\beta_1\}, 
	\end{equation}
	where  $\beta_1$ is the  constant given in Remark \ref{beta}.
\end{lemma}
\begin{proof}
	It follows from \eqref{35} and \eqref{36} that
	\begin{equation*}
		\begin{aligned}
			u_t&=\Delta(\gamma(v)u)-uf(u)\\
			&=\gamma(v)\Delta u +2\gamma'(v)\nabla v \cdot \nabla u +u\gamma''(v)|\nabla v|^2+u\gamma'(v)\Delta v-uf(u)\\
			&=\gamma(v)\Delta u +2\gamma'(v)\nabla v \cdot \nabla u +u\gamma''(v)|\nabla v|^2+u\gamma'(v)(v-u)-uf(u).
		\end{aligned}
	\end{equation*}
	Noting that $v=\mathcal{A}^{-1}[u]$ and $u \leq \|u\|_{L^\infty(\Omega)} $ in $\Omega$, we  use the elliptic comparison principle to obtain that
	\begin{equation}\label{27}
		v(t,x)\leq \|u\|_{L^\infty(\Omega)},~~~~(t,x)\in [0,T_{\max})\times \Omega.
	\end{equation} 
	It  follows from \eqref{gamma2}, \eqref{rangev}, \eqref{27}  and non-negativity of $u$ that
	$$ u\gamma''(v)|\nabla v|^2 +u\gamma'(v)(v-u) \leq u\gamma'(0)(\|u\|_{L^\infty (\Omega)} -u).$$
	Therefore,
	$$u_t-\gamma(v)\Delta u-2\gamma'(v)\nabla v\cdot \nabla u +uf(u) \leq u\gamma'(0) (\|u\|_{L^\infty (\Omega)}-u),~~~~~\text{in~}(0,T_{\max})\times \Omega.$$
	Adding $u$ to both sides of the above equation, it follows from \eqref{44}  that
	\begin{equation*}
		\begin{aligned}
			u_t-\gamma(v)\Delta u-2\gamma'(v)\nabla v\cdot \nabla u +uf(u)+u &\leq u\gamma'(0) (\|u\|_{L^\infty (\Omega)}-u)+u\\
			&\leq \gamma'(0)u(\|u\|_{L^\infty(\Omega)}-u)+uf(u)+\beta_1.
		\end{aligned}
	\end{equation*}
	Therefore, $u$ satisfies
	\begin{equation}\label{23}
		\left\{
		\begin{aligned}
			u_t-\gamma(v)\Delta u-2\gamma'(v)\nabla v\cdot \nabla u +u  &\leq u\gamma'(0) (\|u\|_{L^\infty (\Omega)}-u)+\beta_1,~~~~~\text{in~}(0,T_{\max})\times \Omega,\\
			\nabla u\cdot \textbf{n} &=0, ~~~~\text{on~}(0,T_{\max})\times \partial \Omega,\\
			u(0)&=u^{in},~~~~\text{in~}\Omega.
		\end{aligned}
		\right.
	\end{equation}
	Let $U\in C^1([0,T_{\max}))$ be a solution to the following ordinary differential equation	
	\begin{equation}\label{24}
		\left\{
		\begin{aligned}
			\frac{\rd U }{\;\rd t}+U&=U\gamma'(0) (\|u\|_{L^\infty (\Omega)}-U)+\beta_1,\qquad t\in(0,T_{\max}),\\
			U(0)&=\|u^{in}\|_{L^\infty(\Omega)} .
		\end{aligned}
		\right.
	\end{equation}		
	By the parabolic comparison principle, there holds
	\begin{equation}\label{25}
		u(t,x)\leq U(t),\;\; \;\; (t,x)\in [0,T_{\max})\times \bar{\Omega}.
	\end{equation}
	In particular,  we have $\|u\|_{L^\infty(\Omega)} \leq U(t)$ for all $t \in [0,T_{\max})$. From \eqref{24} and the non-negativity of $\gamma'(0)$, it follows that
	$$\frac{\rd U}{\rd t}+U\leq \beta_1,\;\; \;\; \text{in}\;\; [0,T_{\max}).$$
	Thus,
	\begin{equation}\label{26}
		U(t)\leq U(0)e^{-t}+\beta_1(1-e^{-t})\leq\max\{ \|u^{in}\|_{L^\infty(\Omega)},\, \beta_1\},\;\; \;\; \text{in}\;\; [0,T_{\max}),
	\end{equation}and this completes the proof by \eqref{25}.
\end{proof}

\noindent \textbf{Proof of Theorem \ref{thm2}.}  Once we get \eqref{41}, we can establish the global existence and time-independent boundeness of solutions of problem \eqref{0.1}  based on Theorem \ref{thm2}. 
\section*{Acknowledgments}
Jiang is supported by National Natural Science Foundation of China (NSFC)
under grants No.~12271505 \& No.~12071084,  the Training Program of Interdisciplinary Cooperation of Innovation Academy for Precision Measurement Science and Technology, CAS (No.~S21S3202), and by Knowledge Innovation Program of Wuhan-Basic Research (No.~2022010801010135).
\normalem    
\bibliographystyle{siam}
\bibliography{PSE2023.bib}
\end{document}